\documentclass[12pt]{amsart}
\usepackage{amssymb}

\usepackage{tikz}
\usetikzlibrary{shapes.geometric}
\usetikzlibrary{positioning}
\newtheorem{theorem}{Theorem}
\newtheorem{question}{Question}
\newtheorem{lemma}[theorem]{Lemma}

\newtheorem{corollary}[theorem]{Corollary}

\newtheorem{conjecture}[theorem]{Conjecture}
\usepackage{color}

\newtheorem{definition}[theorem]{Definition}

\newcommand{\ran}{\text{ran}}

\newcommand{\TT}{\mathcal{T}}
\newcommand{\Z}{\mathbb{Z}}
\newcommand{\GG}{\mathcal{G}}

\newcommand{\CC}{\mathcal{C}}
\newcommand{\BB}{\mathcal{B}}

\newcommand{\Ss}{\mathcal{S}}

\newcommand{\upharp}{\upharpoonright}
\newcommand{\N}{\mathbb{N}}

\begin{document}
\title{Locally finite trees and the topological minor relation}
\author{ Jorge Bruno and Paul J. Szeptycki}
\maketitle

\begin{abstract} A well-known theorem of Nash-Williams shows that the collection of locally finite trees under the topological minor relation results in a BQO. Set theoretically, two very natural questions arise:

\begin{enumerate}
\item What is the number $\lambda$ of topological types of locally finite trees?
\item What are the possible sizes of an equivalence class of locally finite trees?
\end{enumerate}
For (1), clearly, $\omega \leq \lambda \leq \mathfrak{c}$ and Matthiesen refined it to $\omega_1 \leq \lambda \leq \mathfrak{c}$. Thus, this question becomes non-trivial when the Continuum Hypothesis is not assumed. In this paper we address both questions by showing that - entirely within ZFC - for a large collection of locally finite trees that includes those with countably many rays: the answer for (1) is $\lambda = \omega_1$, and that for (2) the size of an equivalence class can only be either $1$ or $\mathfrak{c}$.

\end{abstract}

\section{Introduction.} We consider the family of locally finite trees with respect to the topological minor relation where $T\leq^{\sharp} S$ if some subdivision of the tree $T$ embeds as a subgraph of $S$. If two trees are mutually embedded in this way then they are topologically equivalent. 
The topological minor relation is a quasi-order which partially orders the induced equivalence classes. In particular, we recall the result of Nash-Williams that the topological minor relation defines a better-quasi ordering of all trees. Hence, chains in $({\mathcal T},\leq^{\sharp})$ are well-ordered. Matthiesen answered in \cite{M} a question of van der Holst by showing that there are uncountably many topological types of locally finite trees. The first author answered a question from \cite{M} by giving an explicit construction of $\omega_1$ topological types of locally finite trees \cite{B}. This paper addresses some natural related questions concerning the size of equivalence classes and the number of equivalence classes.
 
Let ${\mathcal T}$ denote the set of all countable locally finite rooted trees. Let ${\mathcal G}$ denote the family of all finite rooted trees.  For a rooted tree $T$, the natural order on $T$ that $s\leq t$ if the unique minimal path from the root to $t$ passes through $s$ will be denoted by $\leq$ or $\leq_T$. We refer the reader to \cite{D} for all necessary graph-theoretic background. There are two basic questions addressed here. The first question concerns the number of equivalence classes:

\begin{question} What is the supremum of the lengths of chains in $({\mathcal T},\leq^{\sharp})$? In particular, can there be any strictly increasing well-ordered subsets of $({\mathcal T},\leq^{\sharp})$ of length $\omega_2$ or longer (of course consistent with the negation of CH)?
\end{question}

We are interested in the set of equivalence classes of locally finite trees with respect to the equivalence relation induced by the topological minor relation. So if all chains are of length strictly less than $\omega_2$, then by a result of Kurepa we have a positive answer to the following question: 
\begin{question} Are there are exactly $\omega_1$ equivalence classes of locally finite trees wrt topological equivalence
\end{question}

\noindent {\bf REMARKS.}

The first author gave an explicit construction of a well-ordered strictly increasing chain of length $\omega_1$. This construction produces a sequence of binary trees $T_\alpha$ with countably many branches naturally lexicographically ordered  in type $\alpha$. It is not difficult to extend this chain of trees to obtain  chain of length $\omega_1+\omega+1$ by adding the full n-ary trees on topped off by the maximal locally finite tree (with nth level n-splitting). 

 The first part of this paper is devoted to showing that if we restrict to locally finite trees with countably many branches, then indeed $\omega_1$ is the supremum of lengths of such trees.

In addition, we consider chains involving trees with uncountably many branches and conjecture that for every $\alpha<\omega_2$ there is a chain of length $\alpha$ 

The second half of the paper is devoted to the question about the sizes of the equivalence classes. The Tree Alternative Conjecture concerns the sizes of equivalence classes with respect to the mutual embeddability relation. Indeed, it conjectures that the number of  trees mutually embeddable with a given tree T is either 1 or infinite. We consider a similar question for the relation of the mutual topological minor relation. For a locally finite tree $T$, let $[T]$ denote the equivalence class of topological equivalence and we ask

\begin{question} For a locally finite tree, $T$, is $|[T]|$ either $1$ or ${\mathfrak c}$?
\end{question} 

Note that for the relation of mutually embeddability, there are even countable trees where the number of equivalent trees may by countable or uncountable. 

\vskip 13pt
\section{Lengths of chains}

\vskip 6pt
\noindent Let $R$ be the graph on $\omega$ determined by adding one edge between successive natural numbers (i.e., $n$ is connected to $n-1$ and $n+1$ by an edge). We will refer to this tree as the {\bf ray} or the {\bf single infinite branch}.

Consider ${\mathcal G}^\omega$ the set of all functions $f:\omega\rightarrow {\mathcal G}$. For each such function, let $T_f$ be the tree obtained by attaching a copy of $f(n)$ to $R$ by connecting the root of $f(n)$ to $n$ via an edge. 
  
\begin{lemma} Given $f,g\in {\mathcal G}^\omega$ the following are equivalent:
\begin{enumerate}
\item $T_f \leq^{\sharp} T_g$
\item There is a strictly increasing sequence $(k_n:n\in\omega)$ such that for all $n$ $f(n)\leq^{\sharp} g(k_n)$. 
\end{enumerate}
\end{lemma}
Item (2) in the above lemma suggests the following ordering on ${\mathcal G}^\omega$:

$f\leq^* g$ if there is $(k_n:n\in\omega)$ strictly increasing, such that for all $n,\  f(n)\leq^{\sharp} g(k_n).$ 

This ordering is transitive and reflexive but not antisymmetric. If we declare $f\equiv g$ if both $f\leq^* g$ and $g\leq^* f$ then obtain an equivalence relation. Then $\leq^*$ induces a partial ordering of the set $\{[g]:g\in {\mathcal G}^\omega\}$ of equivalence classes . We will be sloppy and refer to the ordering of $\leq^*$ as a partial ordering of ${\mathcal G}^\omega$. Note that since the topological minor relation is a wqo, by the equivalence given by the above lemma we have that this partial order is well founded and all antichains are finite. Our first goal is to prove 

\begin{theorem}\label{theoremfirst} The set $\{[g]:g\in {\mathcal G}^\omega\}$ of equivalence classes is countable. 
\end{theorem}
By a theorem of Kurepa - which states that for a wqo $P$, $|P|\geq \kappa$ if and only if there is a $\kappa$ well-ordered chain in $P$ - it suffices to prove that this partial order has no uncountable well-ordered subsets. We make a couple of observations about this particular partial order but then we will formulate an prove something more general. 


 We consider three types of functions, those with range that is dominating, those with bounded range, and those with unbounded but not dominating range. First note:

\begin{lemma} Given $f:\omega\rightarrow {\mathcal G}$, if the range of $f$ is dominating in ${\mathcal G}$ in the ordering $\leq^{\sharp}$, then $[f]$ is the maximum in ${\mathcal G}^\omega$. 
\end{lemma}

Indeed for any two functions, $f$ and $g$, if both have range that is dominating then $f$ is equivalent to $g$. 

\begin{lemma} If the range of $f$ is bounded by a single $T$ wrt the topological minor relation, then $\{[g]:g\leq^* f\}$ is countable.   
\end{lemma}

\begin{proof} Suppose that $f$ has range that is bounded by a single finite rooted tree $T$. Then the range of $f$ must be finite. $f$ may be finite to one on some elements of the range, so we may fix $N\in \omega$ such that for $S$ in the range of $f$, $f^{-1}(S)$ is either infinite or contained in $N$. 

Enumerate those elements $S$ of the range of $f$ for which $f^{-1}(S)$ is infinite as $\{T_i:i<k\}$ and let $A_i=f^{-1}(T_i) \setminus N$. So $$f=f\upharpoonright N \cup \bigcup_{i<k}A_i\times\{T_i\}$$

 Now, if $g\leq^*f$ then it follows that there is an $m=m(g)\leq N$ such that 
 \begin{enumerate}
 \item For all $j<m$ there is $l<N$ such that $g(j)\leq^{\sharp}f(l)$, and
 \item For all $S$ in the range of $g\upharpoonright (\omega\setminus m)$ there is $i<k$ such that $S\leq^{\sharp} T_i$
 \end{enumerate}
 Therefore, if $g_1$ and $g_2$ are both below $f$ wrt $\leq^*$, and if $m=m(g_1)=m(g_2)$, and moreover $g_1\upharpoonright m=g_2\upharpoonright m$ {\em and} finally the range of $g_1\upharpoonright \omega\setminus m$ is the same as the range of $g_2\upharpoonright \omega\setminus m$ then by the above observation we have that $g_1\leq^* g_2$ and $g_2\leq^*g_1$ so they are equivalent. 
Now it is clear that there are only countably many equivalence classes $[g]$ with $g\leq^* f$ as required.

\end{proof}

\begin{corollary} The set $\{[f]:\ran(f)\text{ is bounded }\}$ is countable. 
\end{corollary}

\begin{proof}
The set of constant functions in ${\mathcal G}^\omega$ is countable and any bounded function is dominated by a constant function. It follows that there are only countably many equivalence classes of bounded functions. 
\end{proof}

We now consider or more general setting and obtain Theorem \ref{theoremfirst} as a special case. 

Let $(P,\leq_P)$ be an arbitrary wqo. And define a relation $\leq^*$ on $P^\omega$ as above: $f\leq^*g$ if there is a strictly increasing sequence $(k_n)$ such that $f(n)\leq_P g(k_n)$. As above, define $\equiv$ on $P^\omega$ by $f\equiv g$ if $f\leq^*g$ and $g\leq^*f$. Then $\leq^*$ induces a partial ordering of the equivalence classes $\{[f]:f\in P^\omega\}$. We will abuse notation and denote the partial ordering of equivalence classes by $P^\omega$ as well. 
First let $P/\equiv$ be the set of  equivalence classes of $P$ with respect to the equivalence relation $S\equiv T$ in $P$ if $S\leq_P T$ and $T\leq_P S$. 

\begin{theorem} \label{thm:paul} Assume that the set of equivalence classes $P/\equiv$ is countable. Then any increasing chain of length $\omega_1$ in $P^\omega$ is eventually constant. Hence
there are no uncountable well-ordered chains in $P^\omega$. \end{theorem}

\begin{proof} Suppose that ${\mathcal F}=\{f_\alpha:\alpha<\omega_1\}\subseteq P^\omega$ is such that $f_\alpha<^* f_\beta$ for all $\alpha<\beta$. 
Let us say that $T\in P$ is cofinally in the range of ${\mathcal F}$ if for any $n$ there are $\alpha$ and $m>n$ such that $T\leq f_\alpha(m)$, and let
\[
\text{cfrange}({\mathcal F})=\{T:T\text{ is cofinally in the range of }{\mathcal F}\}.
\]

\begin{lemma}\label{cfrangelemma} There is an $\alpha$ such that $\text{cfrange}({\mathcal F})=\text{cfrange}({\mathcal F}\upharpoonright \alpha)$.
\end{lemma}
\begin{proof} Basic closing off argument
\end{proof}

\vskip 6pt

\noindent Fix $\alpha$ as in the above lemma. Then we have the following:
\begin{lemma} For all $\beta\geq \alpha$ there is $n$ such that for all $m>n$ $f_\beta(m)\in \text{cfrange}({\mathcal F})$. 
\end{lemma}

\begin{proof} Suppose not. Then there is an infinite increasing sequence $(n_k)_k$ such that for all $k$ $f_\beta(n_k)=T_k\not\in \text{cfrange}({\mathcal F})$. Now, since $\leq_{P}$ is a well-quasi ordering, we have a subsequence $(T_{k_i})_i$ which is increasing with respect to $\leq_{P}$. But then by definition $T_{k_0}$ is in the $\text{cfrange}({\mathcal F})$. Contradiction. \end{proof}

\begin{lemma}\label{cfrangelemma2} For any $f\in (\text{cfrange}({\mathcal F}))^\omega$, and for any $\beta\geq \alpha$, $f\leq^* f_\beta$.
\end{lemma}

\begin{proof} Fix $f\in (\text{cfrange}({\mathcal F}))^\omega$ and fix $\beta\geq \alpha$. Let us build an increasing sequence $(n_k)_{k\in \omega}$ recursively. Consider $f(0)\in \text{cfrange}({\mathcal F})=\text{cfrange}({\mathcal F}\upharpoonright \alpha)$. Since $f(0)$ is cofinally in the range of ${\mathcal F}\upharpoonright \alpha$, there is a $\beta_0<\alpha$ and an $m_0$ such that $f(0)\leq_{P} f_{\beta_0}(m_0)$. And since $f_{\beta_0}\leq^* f_\beta$ we may fix $n_0$ so that $f_{\beta_0}(m_0)\leq_{P} f_\beta(n_0)$. By transitivity of $\leq_{P}$ we have that $f(0)\leq_{P} f_\beta(n_0)$. 

Now, having constructed $n_0<n_1<...<n_{k-1}$ so that for all $i<k$ we have $f(i)\leq_{P} f_\beta(n_i)$, consider $f(k)$. Since $f(k)$ is cofinally in the range of ${\mathcal F}\upharpoonright \alpha$ we may fix $\beta_k$ and $m_k>n_{k-1}$ such that $f(k)\leq_{P} f_{\beta_k}(m_k)$. Now, $f_{\beta_k}\leq_{P} f_\beta$ and by definition of $\leq_{P}$ we know that there is an increasing sequence $(r_i)$ such that $f_{\beta_k}(i)\leq_{P}f_\beta(r_i)$. Since the sequence of $r_i$'s is increasing, we have that $r_{m_k}\geq m_k$ which was chosen to be larger than $n_{k-1}$. Let $n_k=r_{m_k}$ and note that 
$$f(k)\leq_{P}f_{\beta_k}(m_k)\leq_{P}f_\beta(r_{m_k})=f_\beta(n_k)$$ as required. This completes the construction of the sequence $(n_k)$ which witnesses that $f\leq^*f_\beta$.
\end{proof}

Now, to complete the proof of the Theorem, first note that in the construction above, $m_0$ could have been chosen arbitrarily large and consequentially, since $n_0\geq m_0$ $n_0$ could also be taken arbitrarily large. Thus, if $f\in (\text{cfrange}({\mathcal F}))^\omega$ the sequence $(n_k)_{k\in \omega}$ witnessing $f\leq^* f_\alpha$ can be chosen with $n_0$ as large as we like. To see that there are no $\omega_1$ chains, it suffices to show that there are only countable many distinct equivalence classes among all the $\{[f_\beta]:\beta<\omega_1\}$ and to see this we claim that for all $\beta>\alpha$, $f_\beta$ is equivalent to some finite modification of $f_\alpha$. Since $P/\equiv$ is countable, this suffices. Toward this end, fix $\beta>\alpha$ and choose $N$ large enough so that for all $m\geq N$ $f_\beta(m)$ is cofinally in the range of ${\mathcal F}$. We can do this by Lemma \ref{cfrangelemma} above. 
Consider 
$$
\hat{f}_\alpha = f_\beta\upharpoonright N \cup f_\alpha\upharpoonright (\omega\setminus N).
$$
Then notice that $\hat{f}_\alpha\upharpoonright N=f_\beta\upharpoonright N$, and also, 
$\hat{f}_\alpha\upharpoonright \omega\setminus N$  can be viewed as a function in  $(\text{cfrange}({\mathcal F}))^\omega$. Therefore, by Lemma \ref{cfrangelemma2} there is a sequence $(k_m)_{m\geq N}$ witnessing $\leq ^*$ which can be chosen so that the initial element $k_N\geq N$. Thus, it follows by extending the sequence so that $k_i=i$ for all $i<n$ that $\hat{f}_\alpha\leq^* f_\beta$. 

In the same way we have that $f_\beta\leq^* \hat{f}_\alpha$ and so there are only countably many distinct equivalence classes among the $\{[f_\beta]:\beta\geq\alpha\}$ completing the proof of the theorem.\end{proof}

In fact, the reader can quickly verify that Theorem~\ref{thm:paul} is true for regular cardinals.

\begin{corollary}\label{cor:paulgeneral} Let $\kappa$ be a regular cardinal, $P$ denote a wqo and assume that $|P/\equiv| = \kappa$. Then there are no well-ordered chains in $P^\omega$ of cardinality $>\kappa$; any increasing chain of length $>\kappa$ in $P^\omega$ is eventually constant. 
\end{corollary}


\subsection{Trees with countably many rays}\label{sec:countable}

In the following sections it becomes necessary to forge trees from smaller ones. We do this by {\it joining}, from their root, trees to rays or other types of trees. Therefore, for convenience, from now on when joining a tree $T$, by its root, to a node $s$ of a tree $S$ we write {\bf glue} when the root of $T$ and the node $s$ are fused together, while {\bf attach} refers to the use of an edge for joining them.

Given a tree $T$ and a ray $B$ in $T$, we use the notation $B_n$ to denote a subtree of $T$ that is attached to the $n^\text{th}$ node of $B$ (incompatible with the nodes of $B$ above the $n^\text{th}$ node of $B$). Of course, $B_n$ depends on the node corresponding to the root of $B_n$ that is attached to the $n^\text{th}$ node of $B$, but this will usually not important and clear from the context. Let $S_0$ denote the single vertex tree, $S_1$ be the ray, $S_2$ be the tree composed of the ray with a copy of $S_1$ attached at every node, and in general:
 
 \begin{itemize}
 \item For $\alpha +1 =\beta$: $S_\beta$ is the ray with a copy of $S_\alpha$ attached at every node.  

\medskip

 \item Whenever $\beta$ is limit: select a cofinal $\psi_\beta: \omega \to \beta$ and let $S_\beta$ be comprised of the ray with a copy of $S_{\psi(n)}$ to its $n^\text{th}$ node.
  \end{itemize}
  
  \begin{figure}[h]
\centering
\begin{tikzpicture}[
  triangle/.style = {fill=black!20, regular polygon, regular polygon sides=3 },
    node rotated/.style = {rotate=180},
    border rotated/.style = {shape border rotate=180}, scale=.6,auto=left,every node/.style={scale = 0.4, circle,fill=black!100}]
    
  \node (n1) at (5,10){};
   \node[scale = 1.9, triangle, border rotated] at (9,11.55){$S_\alpha$};
   \node (n2) at (9,10.33){};
   
  \node (n3) at (4,12){};
  \draw (1.8,14) node[scale= 2.1, fill=white,inner sep=2pt]{$s(S_\beta)$};
    \node[scale = 1.9, triangle, border rotated] at (8,13.9){$S_\alpha$};
      \node (n4) at (8,12.66){};
      
  \node (n5) at (3,14){};
    \node[scale = 1.9, triangle, border rotated] at (7,16.2){$S_\alpha$};
    \node (n6) at (7,15){};
    
  \node (n7) at (2,16){};
        \node[scale = 1.9, triangle, border rotated] at (6,18.9){$S_\alpha$};
      \node (n8) at (6,17.66){};

  \node (n9) at (1,18){};
  
  \draw (3,20) node[scale= 2.5, fill=white,inner sep=2pt]{$\vdots$};

  \foreach \from/\to in {n1/n2,n3/n4,n5/n6,n7/n8}
    \draw (\from) -- (\to);

     \foreach \from/\to in {n1/n9}
    \draw[line width=1.4pt] (\from) -- (\to);
\end{tikzpicture}
\caption{The tree $S_\beta$ for $\beta=\alpha+1$ with its spine highlighted. \label{fig:OrderTrees}}
\end{figure}
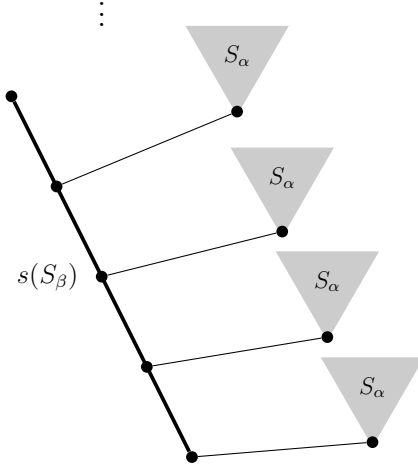

\noindent
Denote $\Ss =\{S_\alpha \mid \alpha<\omega_1\}$. By design, $\alpha<\beta$ implies $S_\alpha$ embeds into $S_\beta$. Notice also that the choice of cofinal function $\psi_\beta$ is irrelevant in the sense of topological embedability. That is, any two such functions yield topologically equivalent trees. For a fixed successor $\alpha+1 = \beta < \omega_1$ (resp. limit $\beta < \omega_1$) we define the \textbf{spine} $s(S_\beta)$ of $S_\beta$ to be the ray where the copies of $S_\alpha$ (resp. $S_\gamma$ with $\psi_\beta(n) = \gamma$ for some $n \in \omega$) are attached to, as illustrated above. It is helpful to note that for every $t$ in the spine of $S_\alpha$, $\{s\in S_\alpha:t\leq_{S_\alpha}s\}$ is topological equivalent to $S_\alpha$. 

\begin{lemma}\label{lem:fullbinary} The collection $\Ss =\{S_\alpha \mid \alpha<\omega_1\}$ contains $\omega_1$ topological types of locally finite trees. Moreover for any tree $T$ the following are equivalent:
\begin{enumerate}
\item  $S_\alpha \leq^{\sharp} T$ for all $\alpha \in \omega_1$, 
\item If $T_2$ is the full binary tree, then $T_2\leq^\sharp T$ 
\item $T$ has uncountably many branches
\end{enumerate}
\end{lemma}
\begin{proof} The proof of this fact can be found in \cite{B}.
\end{proof}

\begin{definition} Define $\TT_{\omega_1} \subset \TT$ to contain all trees with countably many rays and declare $o:\TT_{\omega_1} \to \omega_1$ by
\[
o(T) = sup\{ \alpha:S_\alpha \leq^{\sharp} T\}
\]
Call this function the {\bf order} of a tree. 
\end{definition}

Clearly, finite trees have order $0$ while any infinite tree has order at least $1$. Considering $\omega_1$ and $\bf{\TT_{\omega_1}}$ as categories, the order assignment is functorial where the assignment $\alpha \mapsto S_\alpha$ defines a left adjoint to it.

\begin{lemma}\label{lem:fitcombs} For any $T\in \TT_{\omega_1} $ and $\alpha \in \omega_1$, TFAE:

\begin{enumerate} 
\item $o(T)=\alpha$.
\medskip
\item $S_\alpha \leq^{\sharp} T$ and $S_\gamma \not \leq^{\sharp} T$ for all other $\gamma > \alpha$. 
\medskip
\item $S_\alpha \leq^{\sharp} T$ and from all the possible embeddings of $S_\alpha$ in $T$ there can be only finitely many rays of $T$ containing copies of $s(S_\alpha)$.
\end{enumerate}
\end{lemma}

\begin{proof} $\left[(1) \Leftrightarrow (2)\right]$ Here we must only show that $o(T) = \alpha \Rightarrow S_\alpha \leq^{\sharp} T$, the rest follows by the definition of order. The claim follows easily for a successor $\alpha$, therefore we focus on the limit. For every $\gamma < \alpha$, let $S'_\gamma$ be the subtree of $T$ spanned by one morphism witnessing $S_\gamma \leq^{\sharp} T$. Define $F_\gamma$ to be the subtree of $T$ that results from adding to $S'_\gamma$ the finite ray in $T$ that joins the root of $T$ with the root of $S'_\gamma$. And let $Sp_\gamma$ be the image of the spine of $S_\gamma$ in $S'_\gamma$ along with the finite ray joining the root of $T$ with the root of $S'_\gamma$. Start from the root of $T$ and choose an edge $e_1$ so that for a cofinal in $\alpha$ set $A_1$, all the spines $Sp_\alpha$ from the subtrees in the collection $\{F_\gamma :\gamma\in A_1\}$ contains that edge. This is possible since $T$ is locally finite and contains infinitely many rays. Next, choose an edge $e_2$ adjacent to the vertex of $e_1$ that is not connected to the root of $T$ for which there exists a cofinal, in $\alpha$, $A_2 \subseteq A_1$ so that all spines $Sp_\alpha$ from the subtrees in $\{F_\beta : \beta\in A_2\}$ contain said edge. The same reasons that allowed us to choose $e_1$ also apply to $e_2$. Continue this way and construct the ray $B = e_1e_2\ldots$. We claim that there exists an embedding of $S_\alpha$ into $T$ where the spine of $S_\alpha$ is mapped to $B$. First note that if $\bigcap_n A_n$ is cofinal in $\alpha$ then we are done. Otherwise,
construct $\{\gamma_n:n\in \omega\}$ increasing and cofinal in $\alpha$ and $t\in \omega^\omega$ recursively so that $\gamma_n\in A_{t(n)}\smallsetminus A_{t(n+1)}$ for each $n\in \omega$. 
Therefore, if $S$ is the copy of an $S_\alpha$ obtained by attaching a copy of $S_{\gamma_n}$ to the $n^{th}$ node of a fixed ray, then a subdivision of $S$ embeds into $T$ with its spine mapped onto the constructed ray $e_1e_2.....$. And so $S_\alpha \leq^{\sharp} T$.
 
 \noindent
 $\left[(1) \Leftrightarrow (3)\right]$ A similar argument shows that if there is an infinite number of copies of $S_\alpha$ in $T$ so that the copies of their respective spines encompass more than a finite number of rays in $T$, then one can construct a copy of $S_{\alpha+1}$ in $T$. This a clear contradiction of $o(T) = \alpha$.
\end{proof}

In view of the above, for a ray $B$ of a tree $T$ and an ordinal $\alpha$ write $B \lhd \alpha$ if for all $\beta < \alpha$ there exists $s\in \omega^\omega$ so that $o(B_{s(i)}) \geq \beta$ for all $i\in \N$. We thus obtain the following equivalent characterization of $o(T)$:

\begin{lemma} \label{lem:BB} For any $T\in \TT_{\omega_1} $, $o(T)=0$ if, and only if, $T$ is finite and
\[
o(T) = sup\{\alpha \mid \exists \text{ a ray $B$ of $T$ with } B\lhd \alpha\},
\]
otherwise. Moreover, the set of rays
\[
\BB(T) = \{B \mid sup\{\alpha \mid B\lhd \alpha\} =o(T) \}
\]
\noindent
is non-empty and finite.

\begin{proof} This results follows directly from the previous lemma.
\end{proof}
\end{lemma}

For $1\leq n\in \omega$ and $1\leq \alpha \in \omega_1$ we denote $\TT^n_\alpha$ denote the collection of all locally finite trees $T$ of order at most $\alpha$ such that $|\BB(T)| \leq n$.
\begin{theorem}\label{thm:countablecase} Any chain of trees with at most countably many rays is composed of at most $\omega_1$ topological types.
\end{theorem}

\begin{proof}  We prove by induction on $\alpha$ that for each $n$, any chain composed of trees from $\TT^n_\alpha$ contains at most countably many topological types. Since the locally finite trees are a wqo this also implies that there are only countably many equivalence classes in each $\TT^n_\alpha$.

The case for $n=1$ and $\alpha = 1$ was dealt with in Theorem~\ref{thm:paul}. Notice that if for a fixed $\alpha$ the claim is true of all $\TT^n_\alpha$ then it also holds for $\TT^1_{\alpha+1}$. Indeed, any tree in $T\in \TT^1_{\alpha+1}$ can be identified with a function 
\[
f\in \left(\bigcup_{i\in\omega}\TT^i_{\alpha}\right)^\omega
\]
 as a direct consequence of Lemma~\ref{lem:BB} as follows. If $B \in \BB(T)$ is the unique ray such that $B\lhd \alpha+1$ and if $\{t_n:n\in\omega\}$ are all the nodes of $B$ to which are glued trees of order $\alpha$, and we let $f_T(n)\in \bigcup_{i\in\omega}\TT^i_{\alpha}$ denote the tree glued to $t_n$. Then for any $S\in \TT^1_{\alpha+1}$ we have that $S\leq^\sharp T$ if and only if $f_S\leq^*f_T$. 
Thus, using our inductive hypothesis that there are only countably many topological types in $\bigcup_{i\in\omega}\TT^i_{\alpha}$
and applying Theorem~\ref{thm:paul} we see chains in $T\in \TT^1_{\alpha+1}$ consist of at most countably many topological types.

A similar argument also establishes that if the claim holds true for $\TT^n_\beta$ for all $n$ and all $\beta<\alpha$, then it also holds for $\TT^1_{\alpha}$.  
 
  Next, fix an $\alpha \in \omega_1$ and $N>1$ assume that the result holds for $\TT^m_{\alpha}$ for all $m<N$ and for $\TT^n_\beta$ for all $\beta<\alpha$ and all $n\in \omega$.  
  
 We next show it holds true for $\TT^N_{\alpha}$. Let $C = \{T_\beta \mid \beta \in \omega_1\}$ be a chain of trees from $\TT^{N}_{\alpha}$.
 We can cohesively select, for each $\beta \in \omega_1$, a ray $B_\beta\lhd\alpha$ in $\BB(T_\beta)$ as follows.  
Begin by choosing any ray $B_0$  from $\BB(T_0)$ such that each $B_0\lhd \alpha$. And for all $n<\omega$ use the embeddings a subdivision of $T_n$ into $T_{n+1}$ witnessing $T_n\leq^\sharp T_{n+1}$ to select $B_{n+1}$ recursively as the image of $B_n$ under that embedding. Similarly, at successor stages, choose $B_{\beta+1}$ as the image of $B_\beta$ under the appropriate embedding. 

  At a limit stages $\gamma$, consider the embedding of the subdivision of $T_\beta$ to $T_\gamma$ for $\beta<\gamma$ witnessing $T_\beta\leq^\sharp T_{\gamma}$. $B_\beta$ is mapped to some ray of order $\alpha$ in $T_\gamma$ so we may choose one (of the $N$ such rays), call it $B_\gamma$ so that for cofinal in $\gamma$ number of embeddings $T_\beta\to T_\gamma$ we have that $B_\beta$ is mapped to $B_\gamma$.  

\begin{lemma} For all $\beta<\gamma<\omega_1$ there is an embedding of a subdivision of $T_\beta$ to $T_\gamma$ witnessing $T_\beta\leq^\sharp T_{\gamma}$ such that $B_\beta$ is mapped onto $B_\gamma$
\end{lemma}
\begin{proof} By induction on $\gamma$. For $\gamma$ a successor $\gamma=\eta+1$, by our inductive hypothesis, we have an embedding from a subdivision of $T_\beta$ to $T_\eta$ mapping $B_\beta$ to $B_\eta$. And by definition of $B_{\eta+1}$ there is another such embedding from $T_\eta$ to $T_{\eta+1}$ mapping $B_\eta$ to $B_{\eta+1}$. Taking common subdivisions and composing the embeddings gives an embedding mapping $T_\beta$ to $T_{\eta+1}$ as required. 

For limit $\gamma$ by definition of $B_\gamma$ we may similarly find $\beta'$ with $\beta<\beta'<\gamma$ and an embedding of a subdivision of $T_{\beta'}$ to $T_\gamma$ mapping $B_{\beta'}$ to $B_\gamma$. And as above, using the induction hypothesis on $\beta'$ and by composing embeddings we have the required embedding from $T_\beta$ to $T_\gamma$. 
\end{proof}

Now we proceed as in the previous cases. Let

$$P=\bigcup\{\TT^n_\beta:\beta<\alpha,n\in \omega\}\cup \bigcup_{m<\leq N}\TT^m_\alpha$$

Note first that by our inductive assumption, there are only countably many topological types among the trees in $P$. 

And for each $\beta<\omega_1$ there is an $f_\beta\in P^\omega$ So that each $T_\beta=T_{f_\beta}$ where $T_{f_\beta}$ is the tree obtained by gluing copies of $f_{\beta}(n)$ to the $n^{th}$ node of $B_\beta$. And $T_\beta\leq^\sharp T_\gamma$ implies that $f_\beta\leq^* f_\gamma$. In addition, if also $f_\gamma\leq^* f_\beta$ then $T_\beta$ is topologically equivalent to $T_\gamma$. And so, by Theorem~\ref{thm:paul} we can deduce that the chain $\{T_\beta:\beta<\omega_1\}$ is eventually constant.  This complete the proof.
\end{proof}

The following corollary is a direct consequence of the previous result and Lemma \ref{lem:fullbinary}.

\begin{corollary}\label{cor:splittingleq2} Any chain of trees where each node on every tree has splitting number $\leq 2$ is composed of at most $\omega_1$ topological types.
\end{corollary}

\subsection{Trees with uncountably many rays}\label{sec:uncountable} The answer to Question 1 for chains of trees with uncountably many branches remains open, however, in this section we provide several partial answers. Note, that the previous corollary establishes that chains of binary trees have order type at most $\omega_1+1$ and we begin this section by investigating the length of chains of trees with finitely many nodes of splitting number $\geq 3$. Let us denote by $T_2$ the full binary tree and $T_3$ the full ternary tree, i.e., the infinite tree with all nodes of splitting number $3$.

\begin{lemma}\label{lem:finitebranching3} Let $C$ be a chain of topological types of trees. If there exists a tree $ T \leq^{\sharp} T_3$ with at most finitely many nodes of splitting number $3$ so that $S \leq^{\sharp} T$ for all $S \in C$, then $|C| \leq \omega_1$.
\end{lemma}

\begin{proof} Let $T$ be the tree that results from attaching an extra copy of $T_2$ to the root of $T_2$ itself; all vertices of $T$ have splitting number 2 with the exception of its root, which has a splitting number of 3. Let $C$ be a chain whose top element is $T$. We want from $C$ the upper part composed of trees whose root splitting number is 3 since, as we have shown in Corollary~\ref{cor:splittingleq2}, the bottom part composed of trees with all nodes of order 2 must have a length of $\omega_1$. Denote this new chain by $D$ and let $(S_\alpha:\alpha\in \kappa)$ be its increasing enumeration. Since the root of $T$ and all other trees in the chain $D$ have their root as a special vertex, we will denote the full trees that stem off their roots (i.e., the trees whose root is their ambient trees' root, but only one edge stems from said root) as {\it leaves}. That is, each one of these trees has three leaves stemming from their roots, and each such leaf embeds into $T_2$.

Next, start with the bottom element in $D$, $S_1$, and let $f_1: S_1 \to S_2$ be any one that witnesses $S_1\leq^{\sharp} S_2$. Choose and fix a leaf from $S_1$, call it $L_1$, and let $L_2$ be the leaf in $T_2$ containing $f_1(L_1)$. Do the same for all $\alpha \in \omega$. For each $i \in \omega$, fix an $f_i^\omega: S_i \to S_\omega$ witnessing $S_i \leq^{\sharp} S_\omega$. Choose a leaf from $S_\omega$ for which there exists an infinite set $X \subseteq \{L_n \mid n\in \omega\}$ so that all leaves in $X$ map into it by the $f_i^\omega$'s. It is clear how to proceed for each $\alpha\in |C| = \kappa$: let $f_\alpha :S_\alpha \to S_{\alpha+1}$ be any embedding witnessing $S_\alpha \leq^{\sharp} S_{\alpha+1}$, etc. For the limit case, we use the omega case as a guide. Choose a leaf from $T$ that embeds the chain of leaves we've just constructed. This chain plateaus at stage $\omega_1$. Using the same $f_i$'s and $f_i^\gamma$'s as with $L_1$, repeat the process for the two remaining leaves of $S_1$. We now have three chains of trees that embed into $T_2$ and, thus, they must plateau at stage $\omega_1$. Hence, the leaves of the chain $D$ itself must plateau after $\omega_1$ trees and the chain itself also.

The special node with degree three can be moved around (i.e., it need not be the root of $T$) and yield the same result for the case of $T$ having only one node of branching number $3$. Simple induction does the rest.
\end{proof}

From the proof of Lemma~\ref{lem:finitebranching3} the following is evident.

\begin{corollary} Let $C$ be a chain of locally finite trees. If there exists a tree $T$ with at most finitely many nodes of splitting number $\geq 3$ so that $S \leq^{\sharp} T$ for all $S \in C$, then $C$ is be composed of at most $\omega_1$ topological types.
\end{corollary}

When there is no finite bound to the number of nodes with splitting number $\geq 3$ in a chain of locally finite trees the scenario becomes more complex. Consider, for example, the case of a chain $C$ for which there exists a tree $T$ so that $S \leq^{\sharp} T$ for all $S \in C$, and where $T$ has infinitely many nodes of branching number $\geq 3$ and all such nodes belong to the same ray of $T$. For simplicity, let us assume that each tree $S$ also has an infinite number of nodes of splitting number greater than 2. A moment's thought verifies that each $S \in C$ must have a special ray to which all nodes of degree $\geq 3$ must belong. This creates a scenario much like the one encountered in Theorem~\ref{thm:countablecase}; indeed, each $S\in C$ can be identified with a function
\[
\omega \to \{T\mid T\leq T_2\}.
\]
Consequently, we can employ the ideas in the proof of said theorem in conjunction with Corollary~\ref{cor:splittingleq2} to conclude that $C$ is composed of at most $\omega_1$ topological types of trees. That said, an infinite number of nodes need not span a ray.

\begin{lemma} Any locally finite tree with infinitely many nodes contains a ray. Moreover, if there exists an infinite antichain then the tree contains a copy of the {\it 1-comb} (see Figure~\ref{fig:finitecomb}).
\end{lemma}

\begin{figure}[h]
\centering
\begin{tikzpicture}
  [scale=.8,auto=left,every node/.style={scale = 0.4, circle,fill=black!100}]
  \node (n1) at (1,10){};
  \node (n2) at (0,11) {};
  \node (n3) at (-1,12){};
  \node (n4) at (-2,13){};
  \node (n5) at (-3,14){};
  \node (n6) at (2,11){};
  \node (n7) at (1,12){};
  \node (n8) at (0,13){};
  \node (n9) at (-1,14){};
  \node (n10) at (-2,15){};
  \node (n11) at (-4,15){};
  \draw (-3.2,16) node[scale= 2.5, fill=white,inner sep=2pt]{$\vdots$};

   \node (w1) at (4.5,10){};
  \node (w3) at (3.9,12){};
  \node (w8) at (5.1,12) {};
    \draw (4.5,12) node[scale= 2.5, fill=white,inner sep=2pt]{$\ldots$};
    \draw (4.5,13.3) node[scale= 3.7, fill=white,inner sep=2pt]{$\overbrace{}^{n \text{ nodes}}$};

  \foreach \from/\to in {n1/n5,n1/n6,n2/n7,n3/n8, n4/n9,n5/n10, n5/n11,w1/w3,w1/w8}
    \draw (\from) -- (\to);

\end{tikzpicture}
\caption{The $1$-comb and $V_n$. \label{fig:finitecomb}}
\end{figure}
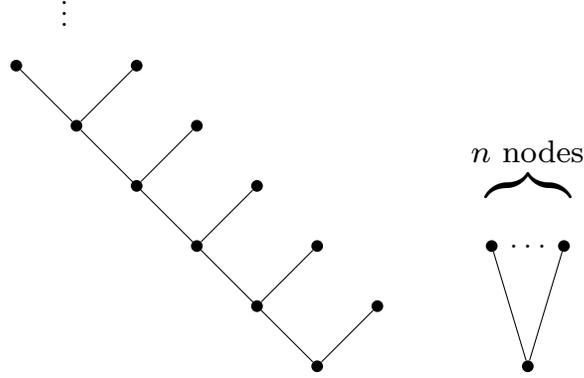

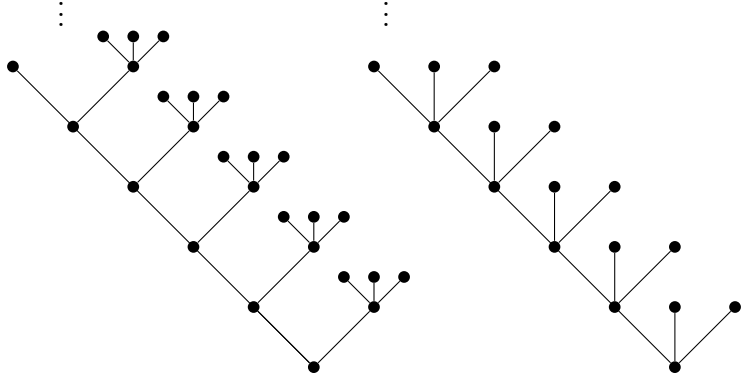
\begin{figure}[h]
\centering
\begin{tikzpicture}
  [scale=.8,auto=left,every node/.style={scale = 0.4, circle,fill=black!100}]
  \node (n1) at (1,10){};
  \node (n2) at (2,11){};
  \node (n3) at (2.5,11.5){};
  \node (n4) at (2,11.5){};
  \node (n5) at (1.5,11.5){};
  
  \node (n6) at (0,11){};
  \node (n7) at (1,12){};
  \node (n8) at (1.5,12.5){};
  \node (n9) at (1,12.5){};
  \node (n10) at (0.5,12.5){};
    
  \node (n11) at (-1,12){};
  \node (n12) at (0,13){};
  \node (n13) at (0.5,13.5){};
  \node (n14) at (0,13.5){};
  \node (n15) at (-0.5,13.5){};
    
  \node (n16) at (-2,13){};
  \node (n17) at (-1,14){};
  \node (n18) at (-0.5,14.5){};
  \node (n19) at (-1,14.5){};
  \node (n20) at (-1.5,14.5){};

  \node (n21) at (-3,14){};
  \node (n22) at (-2,15){};
  \node (n23) at (-1.5,15.5){};
  \node (n24) at (-2,15.5){};
  \node (n25) at (-2.5,15.5){};
 
  \node (n26) at (-4,15){};
  
  \draw (-3.2,16) node[scale= 2.5, fill=white,inner sep=2pt]{$\vdots$};
  
   \node (d1) at (7,10){};
  \node (d2) at (7,11){};
  \node (d3) at (8,11){};

  \node (d4) at (6,11) {};
  \node (d5) at (7,12){};
  \node (d6) at (6,12){};
 
  \node (d7) at (5,12){};
  \node (d8) at (6,13){};
  \node (d9) at (5,13){};
   
  \node (d10) at (4,13){};
  \node (d11) at (5,14){};
  \node (d12) at (4,14){};
  
  \node (d13) at (3,14){};
  \node (d14) at (4,15){};
  \node (d15) at (3,15){};
  
  \node (d16) at (2,15){};

  \draw (2.2,16) node[scale= 2.5, fill=white,inner sep=2pt]{$\vdots$};
     
  \foreach \from/\to in {n1/n6, n1/n2, n2/n3, n2/n4, n2/n5, n1/n26, n6/n7, n7/n8,n7/n9, n7/n10, n11/n12, n12/n13,n12/n14, n12/n15, n16/n17, n17/n18,n17/n19, n17/n20,  n21/n22, n22/n23, n22/n24, n22/n25, d1/d2, d1/d3, d4/d5, d4/d6, d7/d8, d7/d9, d10/d11, d10/d12, d13/d14, d13/d15, d1/d16}
    \draw (\from) -- (\to);

\end{tikzpicture}
\caption{The hairy $3$-comb and the $2$-comb. \label{fig:hairycomb}}
\end{figure}

 For every $n\in \N$, we let $V_n$ be the tree depicted in Figure~\ref{fig:finitecomb}. Denote the tree that result from gluing copies of $V_2$ to every vertex of the ray, the {\bf 2-comb} and the trees that result from attaching a copy of $V_3$ to every node of the ray the {\bf hairy 3-comb} (see Figure~\ref{fig:hairycomb}). Given a tree $T $ let $T^\star_3$ be the subtree of $T$ composed of only the rays of $T$ with the property that for every such ray $B$ there exists an embedding of either the hairy $3$-comb or the $2$-comb into $T$, mapping the only ray from either comb entirely within $B$. Loosely speaking, the rays of $T^\star_3$ trace out all copies of either comb within $T$.

 \begin{lemma} Any chain $C$ of locally finite trees with the property that for every $T\in C$, $T\leq^\sharp T_3$ and $o(T_3^\star) < \omega_1$ is composed of at most $\omega_1$ topological types.
 \end{lemma}
 
 \begin{proof}
  One can readily verify that $S \leq^{\sharp} T$ implies $S_3^\star \leq^{\sharp} T_3^\star$. Next, apply the techniques developed in Section~\ref{sec:countable} to the trees $T_3^\star$.  That is, as a direct consequence of Corollary~\ref{cor:paulgeneral}, any chain of trees $C$ for which every tree $T$ has $o(T_3^\star) =1$ must plateau at height $\omega_1$. Therefore, applying the techniques from Theorem~\ref{thm:countablecase} for the case where $o(T_3^\star)>1$ proves the theorem.
\end{proof}
  
In much the same spirit as with the definitions of the hairy $3$-comb and $2$-comb, for every $n\in \N$ we define {\bf hairy $n$-comb} (resp. {\bf $n$-comb}) to be the tree that results from attaching a copy of $V_n$ (resp. gluing a copy of $V_{n-1}$) to every node of the ray as illustrated in Figure~\ref{fig:ncombs}. Similarly, we define {\bf hairy $\omega$-comb} (resp. {\bf $\omega$-comb}) to be the tree that results from attaching (resp. gluing) a copy of $V_n$ to the $n^{\text{th}}$ node of the ray as illustrated in Figure~\ref{fig:omegacombs}.

 Continuing along this vein, and in line with the definition of $T^\star_3$, for a tree $T$ and a fixed $\alpha\in \omega+1$ we define $T^\star_\alpha$ to be the subtree of $T$ composed of only the rays of $T$ with the property that for every such ray $B$ there exists an embedding of either the hairy $\alpha$-comb or the $\beta$-comb into $T$, mapping the only ray from either comb entirely within $B$ where if $\alpha < \omega$ then $\alpha = \beta+1$ and $\alpha = \beta$, otherwise. 
 
 \begin{lemma} For any tree $T$ and any $n\in \N$, $o(T_n^\star)\leq o(T_{n-1}^\star)$ and $o(T_\omega^\star)\leq o(T_n^\star)$. Moreover, if for all $n\in \N$, $o(T_n^\star) \geq 1$ then $T_\omega^\star \geq 1$.
 \end{lemma}
 
 \begin{proof} For $m>n$, the hairy $n$-comb and $n$-comb embed into the hairy $m$-comb and $m$-comb, respectively. As direct consequence of this embeddability for any tree $T$, any ray in $T_m^\star$ must also be a ray in $T_n^\star$. The last statement follows from a standard compactness argument. 
 \end{proof}

 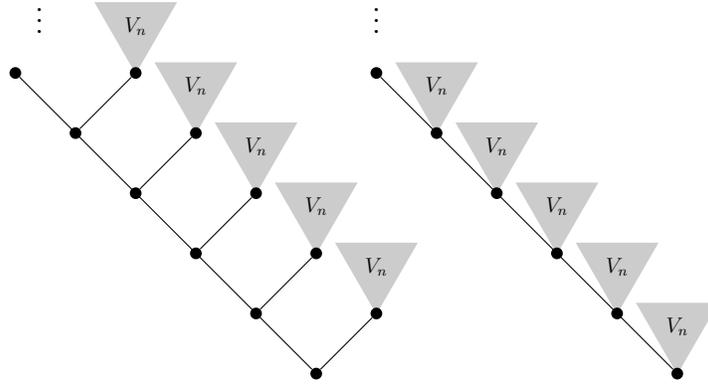
\begin{figure}[h]
\centering
\begin{tikzpicture}[
  triangle/.style = {fill=black!20, regular polygon, regular polygon sides=3 },
    node rotated/.style = {rotate=180},
    border rotated/.style = {shape border rotate=180}, scale=.8,auto=left,every node/.style={scale = 0.4, circle,fill=black!100}]
     \node[scale = 1.4, triangle, border rotated] at (2,11.78){\scalebox{1.2}{$V_{n}$}};
 \node (n1) at (1,10){};
  \node (n2) at (2,11){};

   \node[scale = 1.4, triangle, border rotated] at (1,12.78){\scalebox{1.2}{$V_{n}$}};
  \node (n6) at (0,11){};
  \node (n7) at (1,12){};

     \node[scale = 1.4, triangle, border rotated] at (0,13.78){\scalebox{1.2}{$V_{n}$}};
  \node (n11) at (-1,12){};
  \node (n12) at (0,13){};

     \node[scale = 1.4, triangle, border rotated] at (-1,14.78){\scalebox{1.2}{$V_{n}$}};
  \node (n16) at (-2,13){};
  \node (n17) at (-1,14){};

 \node[scale = 1.4, triangle, border rotated] at (-2,15.78){\scalebox{1.2}{$V_{n}$}};
  \node (n21) at (-3,14){};
  \node (n22) at (-2,15){};
 
  \node (n26) at (-4,15){};
  
  \draw (-3.6,16) node[scale= 2.5, fill=white,inner sep=2pt]{$\vdots$};
  
   \node[scale = 1.4, triangle, border rotated] at (7,10.78){\scalebox{1.2}{$V_{n}$}};
   \node (d1) at (7,10){};
  
  \node[scale = 1.4, triangle, border rotated] at (6,11.78){\scalebox{1.2}{$V_{n}$}};
  \node (d4) at (6,11) {};
 
  \node[scale = 1.4, triangle, border rotated] at (5,12.78){\scalebox{1.2}{$V_{n}$}};
  \node (d7) at (5,12){};
   
    \node[scale = 1.4, triangle, border rotated] at (4,13.78){\scalebox{1.2}{$V_{n}$}};
  \node (d10) at (4,13){};
  
   \node[scale = 1.4, triangle, border rotated] at (3,14.78){\scalebox{1.2}{$V_{n}$}};
  \node (d13) at (3,14){};
  
  \node (d16) at (2,15){};

  \draw (2,16) node[scale= 2.5, fill=white,inner sep=2pt]{$\vdots$};

    \foreach \from/\to in {n1/n2, n1/n26, n6/n7, n11/n12, n16/n17, n21/n22,  d1/d16}
    \draw (\from) -- (\to);

\end{tikzpicture}
\caption{The hairy $n$-comb and the $n$-comb. \label{fig:ncombs}}
\end{figure}

 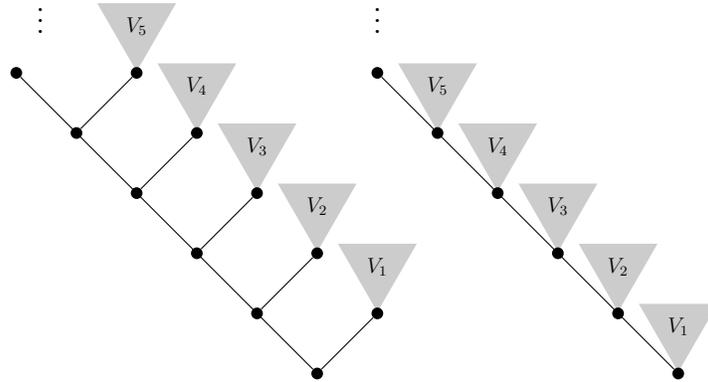
\begin{figure}[h]
\centering
\begin{tikzpicture}[
  triangle/.style = {fill=black!20, regular polygon, regular polygon sides=3 },
    node rotated/.style = {rotate=180},
    border rotated/.style = {shape border rotate=180}, scale=.8,auto=left,every node/.style={scale = 0.4, circle,fill=black!100}]
     \node[scale = 1.4, triangle, border rotated] at (2,11.78){\scalebox{1.2}{$V_{1}$}};
 \node (n1) at (1,10){};
  \node (n2) at (2,11){};

   \node[scale = 1.4, triangle, border rotated] at (1,12.78){\scalebox{1.2}{$V_{2}$}};
  \node (n6) at (0,11){};
  \node (n7) at (1,12){};

     \node[scale = 1.4, triangle, border rotated] at (0,13.78){\scalebox{1.2}{$V_{3}$}};
  \node (n11) at (-1,12){};
  \node (n12) at (0,13){};

     \node[scale = 1.4, triangle, border rotated] at (-1,14.78){\scalebox{1.2}{$V_{4}$}};
  \node (n16) at (-2,13){};
  \node (n17) at (-1,14){};

 \node[scale = 1.4, triangle, border rotated] at (-2,15.78){\scalebox{1.2}{$V_{5}$}};
  \node (n21) at (-3,14){};
  \node (n22) at (-2,15){};
 
  \node (n26) at (-4,15){};
  
  \draw (-3.6,16) node[scale= 2.5, fill=white,inner sep=2pt]{$\vdots$};
  
   \node[scale = 1.4, triangle, border rotated] at (7,10.78){\scalebox{1.2}{$V_{1}$}};
   \node (d1) at (7,10){};
  
  \node[scale = 1.4, triangle, border rotated] at (6,11.78){\scalebox{1.2}{$V_{2}$}};
  \node (d4) at (6,11) {};
 
  \node[scale = 1.4, triangle, border rotated] at (5,12.78){\scalebox{1.2}{$V_{3}$}};
  \node (d7) at (5,12){};
   
    \node[scale = 1.4, triangle, border rotated] at (4,13.78){\scalebox{1.2}{$V_{4}$}};
  \node (d10) at (4,13){};
  
   \node[scale = 1.4, triangle, border rotated] at (3,14.78){\scalebox{1.2}{$V_{5}$}};
  \node (d13) at (3,14){};
  
  \node (d16) at (2,15){};

  \draw (2.0,16) node[scale= 2.5, fill=white,inner sep=2pt]{$\vdots$};

    \foreach \from/\to in {n1/n2, n1/n26, n6/n7, n11/n12, n16/n17, n21/n22,  d1/d16}
    \draw (\from) -- (\to);

\end{tikzpicture}
\caption{The hairy $\omega$-comb and the $\omega$-comb. \label{fig:omegacombs}}
\end{figure}

\begin{corollary} Let $C$ be a chain of trees so that for all $T \in C$ and $\alpha\in \omega+1$, $o(T^\star_\alpha) < \omega_1$. Then $C$ is composed of at most $\omega_1$ topological types of locally finite trees.
\end{corollary}

The above result is the strongest we are able to prove. Scenarios with chains where for some tree $T$ and $n\in \N$, $o(T^\star_\alpha) \geq \omega_1$ are unknown to the present authors. Even some of the simplest of cases remain elusive. For example, let $\CC$ denote the collection of trees forged as follows: take $T_2$ and subdivide each edge in two by adding a new node. To every new node (i.e., every node of splitting number 1) attach a finite tree.

\begin{question} What is the size of $[\CC]$?
\end{question}

 We finish this section by remarking that while we are able to construct chains longer than $\omega_1$, e.g., it is actually possible to construct a chain of inequivalent topological types of locally finite trees of order type $\sum_{n\in \omega+1} \omega_1^n$ by employing trees $T$ with $o(T_\alpha^\star) < \omega_1$, for all $i\in \omega+1$, we do not know the answer to the following questions:
 
 \begin{question} Is there, for each $\alpha<\omega_2$ a chain of inequivalent topological types of locally finite trees of order type $\alpha$?
\end{question}
\begin{question} Is it consistent that there is a chain of inequivalent topological types of locally finite trees of order type $\omega_2$ or longer (of course in the presence of the negation of the Continuum Hypothesis)?
\end{question}


\section{Size of equivalence classes} 
The Tree Alternative Conjecture concerns the sizes of equivalence classes with respect to mutual embeddable relation. Indeed, it conjectures that the number of isomorphism classes of trees mutually embeddable with a given tree T is either 1 or infinite. We consider a similar question for the relation of the mutual topological minor relation and denote $T\equiv^{\sharp} S$ if both $T\leq^{\sharp} S$ and $S\leq^{\sharp} T$. While in the previous part of this paper we were analyzing rooted trees, some of that analysis is helpful in proving the topological minor variant of the tree alternative conjecture: 

\begin{conjecture} For a given locally finite tree $T$, the number of isomorphism classes of trees that are mutually topological minor with $T$ is either $1$ or ${\mathfrak c}$. 
\end{conjecture}

We prove the conjecture for locally finite trees with at most countably many infinite rays. 

\begin{theorem}\label{main} Suppose that $T$ is a locally finite tree with at most countably many infinite rays, then the number of isomorphism classes of trees that are mutually topological minor with $T$ is either $1$ or ${\mathfrak c}$. 
\end{theorem}

To cobble together a proof we need some observations. In all the discussions below ``tree'' means ``unrooted locally finite tree.'' 

In the earlier section where the order of a tree was defined, it was implicit that this was a property of rooted trees. However, the order of a tree is independent of which node is taken as the root:

\begin{lemma} Suppose $T$ is a tree with $r_0$ and $r_1$ are two nodes of $T$. Then the order of the rooted tree $(T,r_0)$ is the same as the order of $(T,r_1)$. Moreover, if $B$ is an infinite ray of $T$ with initial node $r_0$ and if $B'$ is a ray of $T$ with initial node $r_1$ and if either $r_0\in B'$ or $r_1\in B$ then the order of $B$ in $(T,r_0)$ is the same as the order of $B'$ in $(T,r_1)$,
\end{lemma}

Hence, the order of an unrooted tree and the order of an infinite ray in an unrooted tree are well defined by fixing any node in the tree (or in the ray) and calculating the order in the associated rooted tree. 

\begin{lemma} Suppose that $T$ and $S$ are trees, $G$ a subdivision of $T$ and $f:G\rightarrow S$ an embedding witnessing that $T\leq^{\sharp} S$. Suppose further that $B$ is an infinite ray of $T$ and that $f(B)$ is the infinite ray of $S$ induced by the image of $B$ in $S$. Then the order of $B$ is less than or equal the order of $f(B)$. Hence the order of $T$ is less than or equal the order of $S$. 
\end{lemma}

If $T\leq^{\sharp} S$, and if $G$ is the subdivision and $f$ is the embedding of $G$ into $S$ witnessing this, we will suppress explicit reference to $G$ and refer only to the embedding $f$ as restricted to $T$ and call it the embedding of $T$ into $S$ witnessing $T\leq^{\sharp} S$. 

\begin{corollary} Suppose that $T\equiv^\sharp S$. Then $o(T)=o(S)$ and both have the same number of rays of maximal order. Moreover if $f$ is an embedding witnessing $T\leq^\sharp S$ or $S\leq^\sharp T$,  then $f$ maps a ray of maximal order of one to a ray of maximal order of the other. 
\end{corollary}

Before proceeding, let us first describe those trees $T$ with exactly one isomorphism class that is mutually topologically minor to $T$. First note that forfinite trees $T$ and $S$ $T\equiv^\sharp S$ if and only if $T$ is isomorphic to $S$. Indeed, one can first note that if $T\equiv^\sharp S$ then $T$ and $S$ have the same number of nodes, and then proceed by induction on the number of nodes. But there are infinite trees with exactly one isomorphism class. Recall, $T$ has order 1 if and only if $T$ is infinite and consists of a finite initial tree $G$ with finitely many infinite paths connected to some nodes of $G$.

\begin{theorem}\label{1class} Suppose $T$ is a tree and $o(T)\leq 1$, then $T\equiv^\sharp S$ if and only if $T$ is isomorphic to $S$. 
\end{theorem}

And the next theorem completes the proof of Theorem \ref{main}.

\begin{theorem} Suppose $T$ is a such that $o(T)=\alpha>1$. Then there are continuum many trees all topologically equivalent to $T$ but pairwise non-isomorphic. 
\end{theorem}

\begin{proof} Given a graph $G$, let us say that a sequence of nodes $x_0,x_1,...x_n$ is a simple path if each $x_i$ is connected to $x_{i+1}$ by an edge and for every $0<i<n$ the degree of the vertex $x_i$ is the minimal possible, i.e., $d(x_i)=2$. 

Given a graph $G$, by a {\bf collapse of $G$}, we mean a graph formed by fixing a collection of simple paths in $G$ and replacing each simple path $(x_0,x_1,....x_n)$ in that collection by a single edge connecting $x_0$ to $x_1$ (so deleting all the nodes $x_1,...x_{n-1}$ from $G$ in the process). Clearly $G$ is a subdivision of any of its collapses so any collapse of $G$ is a topological minor of $G$.  

Now fix $T$ of order $\alpha>0$ and let $\{B_k:k<n\}$ enumerate all infinite rays of order $\alpha$ in $T$.  Let us enumerate the nodes of each $B_k$ as $\{b_k(i):i<\omega\}$ as described above. For each $k$ and each $i>0$ let $T^k(i)$ denote the induced subgraph of $T$ that contains $b_k(i)$ and all nodes that can be connected to $b_k(i)$ by a path disjoint from the the ray $B_k$. I.e., $T^k(i)$ is the subtree attached to the node $b_k(i)$ off of the ray $B_k$.  By  truncating the rays' initial segments we may assure a few things:
\begin{enumerate}
\item $B_k$ is disjoint from $B_j$ for all $k\not=j <n$
\item For all $k$ and for all $i$, there is $m>i$ such that $T^k(i)\leq^{\sharp} T^k(m)$.
\end{enumerate}
We now form a collapse of $T$ by taking all maximal simple paths contained in the rays $B_k$. I.e., we delete all nodes of the rays $B_k$ for which $T_k(i)$ is empty. 
We denote this collapse of $T$ by $S$. Since $T$ is a subdivision of $S$ we have $S\leq^{\sharp} T$. 
\begin{lemma} $T\leq^{\sharp} S$
\end{lemma}

\begin{proof} We need to embed a subdivision of $T$ into $S$. First fix a maximal subtree of $T$ that contains the set of initial nodes $I=\{b_k(0):k<n\}$. from all the rays $B_k$ for $k<n$. Let this be the subtree induced by the set of all nodes that can be connected to some node in the set $I$ via a path disjoint from $\bigcup_k B_k$. On this subtree let we let the embedding be the identity. For the rest of the tree, we work with each ray $B_k$ separately. Fix $k<n$ and fix an sequence $\{s_k(i):i<\omega\}$ so that for all $i$ we have $T^k(i)\leq^{\sharp} T^k(s_k(i))\not=\emptyset$ (note that since $T^k(s_k(i))\not=\emptyset$ it follows that this subtree and the node $b_k(s_k(i))$ is in $S$). We may not fix an embedding of a subdivision of $T^k(i)$ into $T^k(s_k(i))$ and by adding sufficiently many nodes between $b_k(i)$ and $b_k(i+1)$ for each $i$, extend this to an embedding of a subdivision of the entire subtree living on $B_k$ into the subtree of $S$ living on the same ray. Since the rays are pairwise disjoint, we can join all these together into a single subdivision of tree and embedding of it into $S$, witnessing $T\leq^{\sharp} S$. 

\end{proof}


For a function $f:\omega\rightarrow \omega$ let us define $S_f$ to be the subdivision of $S$ obtained by replacing the edge connecting each successive elements $b_k(i)$ and $b_k(i+1)$ by a path of length $f(i)$. Since the rays are disjoint, this subdivision is well defined. 

\begin{lemma} Suppose that $f, g\in \omega^\omega$ and suppose that $S_f$ is isomorphic to $S_g$. Then the range of $f$ is, modulo a finite set, equal to the range of $g$. I.e. $ran(f)\setminus ran(g)$ and $ran(g)\setminus ran(f)$ are both finite. 
\end{lemma} 
\begin{proof} If two trees are isomorphic, they must be of the same order and they must map infinite rays to infinite rays preserving the order of the rays. Therefore, there are two rays $B_i$ and $B_j$ with $i$ and $j$ less than $n$ such that modulo some finite initial segment of each ray, the $f$-subdivision of $B_i$ must be mapped isomorphically onto the $g$-subdivision of $B_j$. Nodes of degree 2 must be preserved in this map so the lengths of the subdivisions between the successive nodes of $B_i$ must coincide with the lengths of the subdivisions between successive nodes of $B_j$ except for at most a finite initial segment of each ray. Therefore the range of $f$ is almost equal the range of $g$.  
 
\end{proof}

Now to complete the proof of the Theorem, it suffices to fix continuum many functions which pairwise have distinct ranges modulo $=^*$. 
\end{proof}

Just as with Section~\ref{sec:uncountable}, trees with uncountably many rays pose a largely more complicated problem. However, it is not difficult to show that employing the methods developed above, the following partial result hold. 

\begin{theorem} Let $T$ be any locally finite tree for which $o(T^\star_\alpha)<\omega_1$ for all $\alpha \in \omega+1$. Then $|[T]|$ is either $1$ or ${\mathfrak c}$.
\end{theorem}

\begin{proof} The idea here it to replace, in the above results, the branches of a tree $T$ with the branches of $T_\alpha^\star$ where $\alpha = \text{sup}\{n\in \omega+1 \mid o(T_n^\star) > 0\}$.
\end{proof}


\begin{thebibliography}{99}

\bibitem{B} Bruno, J. ``A family of $\omega_1$ topological types of locally finite trees'', Discrete Mathematics, v 340, pp. 794-795, 2017. 
\bibitem{D} Diestel, R. {\bf Graph Theory}, Graduate Texts in Mathematics v 173, Springer 2010.
\bibitem{M} Matthiesen, L. ``There are uncountably many topological types of locally finite trees'' Journal of Combinatorial Theory, Series B Volume 96, Issue 5, September 2006, Pages 758-760.
\end{thebibliography}
\end{document}